\documentclass[11pt]{amsart}
\usepackage{amsfonts,amscd,latexsym,amsmath,amssymb,enumerate,verbatim,amsthm,epsfig,alltt,enumitem,color}

\newcommand{\Nat}{\mathbb{N}}
\newcommand{\Int}{\mathbb{Z}}

\newcommand{\sep}{\mathrm{sep}}

\newcommand{\NN}{\mathcal{N}}
\newcommand{\UU}{\mathcal{U}}

\theoremstyle{plain}

\newtheorem{theorem}{\bf Theorem}[section]
\newtheorem{lemma}[theorem]{\bf Lemma}
\newtheorem{proposition}[theorem]{\bf Proposition}
\newtheorem{corollary}[theorem]{\bf Corollary}

\newtheorem{claim}[theorem]{\bf Claim}

\newtheorem{thmx}{Theorem}

\theoremstyle{definition}
\newtheorem{definition}[theorem]{\bf Definition}

\newtheorem{remark}[theorem]{\bf Remark}

\begin{document}
\title[Garden of Eden and weakly periodic points]{Garden of Eden and weakly periodic points for certain expansive actions of groups}
\author[M. Doucha]{Michal Doucha}
\address{Institute of Mathematics\\
	Czech Academy of Sciences\\
	\v Zitn\'a 25\\
	115 67 Praha 1\\
	Czech Republic}
\email{doucha@math.cas.cz}
\keywords{expansive dynamical systems, Garden of Eden theorem, entropy, periodic points, amenable groups, ends of groups}
\subjclass[2020]{37B05, 37B40, 37B65, 37C25, 37C85, 43A07, 20E06}
\thanks{The author was supported by the GA\v{C}R project 19-05271Y and RVO: 67985840.}

\begin{abstract}
We present several applications of the weak specification property and certain topological Markov properties, recently introduced by S. Barbieri, F. Garc\' ia-Ramos and H. Li, and implied by the pseudo-orbit tracing property, for general expansive group actions on compact spaces.

First we show that any expansive action of a countable amenable group on a compact metrizable space satisfying the weak specification and strong topological Markov properties satisfies the Moore property, i.e. every surjective automorphism of such dynamical system is pre-injective. This together with an earlier result of H. Li (where the strong topological Markov property is not needed) of the Myhill property, which we also re-prove here, establishes the Garden of Eden theorem for all expansive actions of countable amenable groups on compact metrizable spaces satisfying the weak specification and strong topological Markov properties. We hint how to easily generalize this result even for uncountable amenable groups and general compact, not necessarily metrizable, spaces.

Second, we generalize the recent result of D. B. Cohen that any subshift of finite type of a finitely generated group having at least two ends has weakly periodic points. We show that every expansive action of such a group having a certain Markov topological property, again implied by the pseudo-orbit tracing property, has a weakly periodic point. If it has additionally the weak specification property, the set of such points is dense.
\end{abstract}

\maketitle

\section*{Introduction}
Let $G$ be a group and $A$ a finite (discrete) set. Then $A^G$, with the product topology, is a compact topological space on which $G$ naturally acts by homeomorphisms. Such a topological dynamical system is called a (topological) shift and any of its (topologically) closed $G$-invariant subsets are called subshifts. These are the objects of study of symbolic dynamics, whose many results and notions have inspired and have been generalized to more general dynamical systems.

One of those is the Garden of Eden theorem. For $G=\Int^d$, Moore in \cite{Moore}, resp. Myhill in \cite{Myhill} proved that every continuous surjective $G$-equivariant map $\tau:A^G\to A^G$ is pre-injective, resp. conversely every such pre-injective map is surjective (we shall define pre-injectivity later). A dynamical system for which this equivalence holds is said to satify the \emph{Garden of Eden theorem}. This result was generalized to all amenable groups in \cite{CSMS99} and later also for certain subshifts of amenable groups (see \cite{Fio03}, \cite{CSCo12}). It didn't take long for the Garden of Eden theorem to be considered for many other dynamical systems (see e.g. \cite{CSCo16}, \cite{CSCo06},\cite{CSCoLi19},\cite{CSCoPh}, and \cite{Li19}). One of the most general versions of the Myhill property was considered by H. Li in \cite{Li19} where he proves that every expansive action of an amenable group on a compact metrizable space having the weak specification property satisfies the Myhill property. 

Our aim is to provide such a general result also for the Moore property. However, the weak specification property alone is not enough for establishing this property. This is known even for subshifts, where a counterexample was provided for a strongly irreducible subshifts in \cite{Fio00} and it was showed in \cite{Li19} that the weak specification property for subshifts is equivalent to being strogly irreducible. We add one of the topological Markov properties that were introduced very recently by S. Barbieri, F. Garc\' ia-Ramos, and H. Li in \cite{BaFRLiarxiv} and which generalizes the pseudo-orbit tracing property, also known as shadowing, which for subshifts corresponds to being of finite type (see \cite{ChKeo}).

Therefore, one of our main results is the following. All the notions will be defined in the next section.
\begin{thmx}
Let $G$ be a (countable) amenable group acting continuously and expansively on a compact metrizable space $X$ so that the action has the weak specification and strong topological Markov properties (recall again that it is implied by the pseudo-orbit tracing property). Then the dynamical system $(X,G)$ has the Moore property.
This together with the result of H. Li, which we also re-prove, implies that this dynamical system satisfies the Garden of Eden theorem.
\end{thmx}
The restriction to countable groups and metrizable spaces is not at all necessary, although perhaps the most interesting. In Subsection~\ref{subsect:nonmetrizable} we hint on how to easily generalize the result also for uncountable groups and non-metrizable spaces. This simplifies and generalizes the result of Ceccherini-Silberstein and Cornaert from \cite{CSCo21}.\bigskip

Next we consider periodic, resp. weakly periodic points of dynamical systems. If $(X,G)$ is a dynamical system, we call a point $x\in X$ \emph{periodic} if its orbit is finite, or equivalently, if its stabilizer has a finite index in $G$. We call a point $x\in X$ \emph{weakly periodic} if its stabilizer is infinite. In a recent breakthrough \cite{Cohen}, Cohen showed that if $G$ is a finitely generated group having at least two ends, then every subshift of finite type of $G$ must contain a weakly periodic point. Compare this with the other results for several one-ended groups which admit a subshift of finite type on which the group acts freely (see e.g. the classical result for $\Int^2$ in \cite{Berger} and many other more recent developments in \cite{CuKa}, \cite{Mozes}, \cite{Cohen17}).

Here we generalize Cohen's result again to far more general dynamical systems. For this, we introduce a new topological Markov property, called here \emph{the cover strong topological Markov property}, in the spirit of \cite{BaFRLiarxiv}, which is implied by the pseudo-orbit tracing property and implies the uniform strong topological Markov property - we do not know if it is equivalent to the latter. We show the following.

\begin{thmx}
Let $G$ be a finitely generated group having at least two ends. Suppose that $G$ acts continuously and expansively on a compact metrizable space $X$ so that the action has the cover strong topological Markov property. Then $(X,G)$ has weakly periodic points.

If the action has moreover the weak specification property, then the set of such points is dense in $X$.
\end{thmx}

For examples of dynamical systems satisfying various topological Markov properties, the pseudo-orbit tracing property, and of systems satisfying them together with the weak specification property, we refer the reader e.g. to \cite{BaFRLiarxiv} and \cite{CSCoLi21}. We remark that among the standard examples are, besides various subshifts, also various finitely presented algebraic actions.

\section{Preliminaries}
Throughout the paper we work with groups acting continuously on compact (most of the time metrizable) spaces. Groups, ususally denoted by $G$, will be always discrete and will be implicitly assumed to be countable, although this assumption is usually either not necessary, or the arguments can be easily modified to work for uncountable groups as well. Topological spaces with a continuous action of a group $G$ are sometimes called \emph{$G$-spaces}.

When working with a compact metrizable space, we fix some compatible metric and formulate most of the notions using such a metric. For that reason, many of the statements are formulated for compact metric spaces instead of compact metrizable spaces. 
However, it should be emphasized that the statements are almost always of a topological, not metric, nature, so the choice of the compatible metric is irrelevant. This can be easily verified e.g. in the following definition, where we specify which topological dynamical systems we shall work with.

\begin{definition}\label{def:expansiveness}
Let $G$ be a group acting continuously on a compact metrizable space $X$. We say that the action is \emph{expansive} if, having fixed some compatible metric $d$ on $X$, there exists $\delta>0$ (called the \emph{expansiveness constant}) such that for every $x\neq y\in X$ there is $g\in G$ so that \[d(gx,gy)>\delta.\]
\end{definition}

Let $X$ be a compact $G$-space and $d$ a continuous pseudometric on $X$. For every subset $E\subseteq G$ we shall denote by $d_E$ the pseudometric \[d_E(x,y):=\sup_{g\in E} d(gx,gy).\]
\begin{lemma}\label{lem:keyexpansivefact}
	Let $G$ act continuously on a compact metric space expansively, with an expansive constant $\delta>0$. Then for every $0<\gamma$ there exists a finite set $D_\gamma\subseteq G$ such that for every $x,y\in X$ if $d(x,y)\geq \gamma$ then $d_{D_\gamma}(x,y)>\delta$.
\end{lemma}
\begin{proof}
	Suppose that no such finite set exists. Then for every finite $D\subseteq G$ there exist elements $x_D,y_D\in X$ such that $d(x_D,y_D)\geq \gamma$, yet $d_D(x,y)\leq \delta$. Without loss of generality, we may assume that $x_D\to x\in X$ and $y_D\to y\in X$ (where $(x_D)$ and $(y_D)$ are nets indexed and ordered by finite subsets of $G$ ordered by inclusion) By continuity, $d(x,y)\geq \gamma$ and $d_D(x,y)\leq \delta$ for every finite $D\subseteq G$. The first inequality implies that $x\neq y$, while the latter, using expansiveness, that $x=y$, a contradiction.
\end{proof}
The following, formally stronger, result follows immediately from Lemma~\ref{lem:keyexpansivefact}.
\begin{proposition}\label{prop:expansiveconstant}
	Let $G$ act continuously and expansively on a compact metric space expansively, with an expansive constant $\delta>0$. Then for every subset $S\subseteq G$, $\gamma>0$ and $x,y\in X$, if $d_S(x,y)\geq \gamma$, then $d_{D_\gamma\cdot S}(x,y)>\delta$.
\end{proposition}
\begin{definition}
	Let $G$ act continuously on a compact metric space $X$. For $x,y\in X$ and $\varepsilon>0$, we denote by $\Delta_\varepsilon(x,y)$ the set $\{g\in G\colon d(gx,gy)>\varepsilon\}$. If the constant $\varepsilon$ is fixed and clear from the context, we write simply $\Delta(x,y)$.
\end{definition}
Notice that if the action from the previous definition is moreover expansive with the expansiveness constant $\delta>0$ and $\varepsilon\leq \delta$, then for $x,y\in X$, $\Delta_\varepsilon(x,y)$ is non-empty if and only if $x\neq y$.

Next we introduce a crucial notion related to the Garden of Eden theorem.
\begin{definition}
	Let $G$ act continuously on a compact metric space $X$. We say that two elements $x,y\in X$ are \emph{homoclinic}, $x\sim y$ in symbols, if $\lim_{g\to\infty} d(gx,gy)=0$.
\end{definition}
It is easy to see that being homoclinic is an equivalence relation.
\begin{corollary}\label{cor:Deltahomoclinic}
	Let $G$ act continuously and expansively on a compact metric space $X$. Then for every $\delta>0$ and every elements $x,y\in X$ we have $x\sim y$ if and only if $\Delta_\delta(x,y)$ is finite.
\end{corollary}
\begin{proof}
	Fix $\delta>0$ and $x,y\in X$. If $x\sim y$, then by definition $\lim_{g\to\infty} d(gx,gy)=0$, so $\Delta_\delta(x,y)$ is finite. Conversely, suppose that $\Delta_\delta(x,y)$ is finite, however $d(gx,gy)$ does not converge to $0$. Then there is $\gamma>0$ and an infinite set $S\subseteq G$ such that for all $s\in S$, $d(sx,sy)\geq \gamma$. By Lemma~\ref{lem:keyexpansivefact}, for each $s\in S$, $D_\gamma s\cap \Delta_\gamma(x,y)\neq\emptyset$. Since $S$ is infinite and $D_\gamma$ is finite, this implies that $\Delta_\delta(x,y)$ is infinite, a contradiction.
\end{proof}

We also recall the following result, likely well known, that the computation of topological entropy is easier for expansive dynamical systems.
\begin{proposition}\label{prop:entropyexpansive}
	Let $G$ be an amenable group and let $X$ be a compact expansive $G$-metric space, with an expansive constant $\delta>0$. Then $h_\mathrm{top}(X,G)=h_\sep(\delta,d)$.
\end{proposition}
\begin{proof}
	Recall that $h_\mathrm{top}(X,G)=\sup_{\varepsilon>0} h_\sep(\varepsilon,d)$, where\\ $h_\sep(\varepsilon,d)=\limsup_{n\to\infty}\frac{1}{|F_n|} \log \sep(d_{F_n},\varepsilon)$ and $(F_n)_n$ is an arbitrary F\o lner sequence (or net if $G$ is uncountable). So it suffices to show that for every $\varepsilon>0$ we have $h_\sep(\varepsilon,d)=h_\sep(\delta,d)$. Since for $0<\varepsilon'<\varepsilon$ we have $h_\sep(\varepsilon',d)\geq h_\sep(\varepsilon,d)$, it suffices to show that for every $0<\gamma<\delta$ we have $h_\sep(\gamma,d)\leq h_\sep(\delta,d)$.
	
	Fix $0<\gamma<\delta$. By Proposition~\ref{prop:expansiveconstant} there is a finite set $D_\gamma\subseteq G$ such that for every subset $S\subseteq G$ and $x,y\in X$ we have $d_{D_\gamma\cdot S}(x,y)\geq \delta$ if $d_S(x,y)\geq \gamma$. It follows that for every subset $S\subseteq G$, $\sep(d_{D_\gamma\cdot S},\delta)\geq \sep(d_S,\gamma)$. Fixing a F\o lner sequence (resp. net) $(F_n)_n$ we get \[\limsup_{n\to\infty}\frac{1}{|F_n|} \log \sep(d_{F_n},\gamma)\leq \limsup_{n\to\infty}\frac{1}{|F_n|} \log \sep(d_{D_\gamma\cdot F_n},\delta).\] Since $(F_n)_n$ is F\o lner, so for every finite subset $D\subseteq G$ (in particular for $D-D_\gamma$) we have $\lim_{n\to\infty} \frac{|D\cdot F_n\Delta F_n|}{|F_n|}=0$, we get
	\begin{enumerate}
		\item $(D_\gamma\cdot F_n)_n$ is F\o lner as well,
		\item $
		\begin{aligned}[t]\limsup_{n\to\infty}\frac{1}{|F_n|} \log \sep(d_{D_\gamma\cdot F_n},\delta) =\\
		=\limsup_{n\to\infty}\frac{1}{|D_\gamma\cdot F_n|} \log \sep(d_{D_\gamma\cdot F_n},\delta).\end{aligned}
		$ 
	\end{enumerate}
	
	This implies that $h_\sep(\gamma,d)\leq h_\sep(\delta,d)$ and we are done.
\end{proof}

\begin{definition}\label{def:memoryset}
	Let $X$ and $Y$ be compact metric spaces on which a group $G$ acts continuously and expansively with expansiveness constants $\delta_X>0$ and $\delta_Y>0$ respectively. Let $\tau:X\rightarrow Y$ be a continuous $G$-equivariant map. A \emph{memory} set for $\tau$ is a set $M\subseteq G$ such that for every $x,y\in X$ and $g\in G$, if $d_M(gx,gy)\leq \delta_X$ then $d(g\tau(x),g\tau(y))\leq \delta_Y$.
\end{definition}
\begin{lemma}\label{lem:memoryset}
	Let $X$, $Y$, and $\tau:X\rightarrow Y$ be as above. There exists a finite memory set for $\tau$.
\end{lemma}
\begin{proof}
	Suppose the contrary. Then for every finite set $E\subseteq G$ there are $x_E,y_E\in X$ and $g_E\in G$ such that $d_E(g_E x_E,g_E y_E)\leq \delta_X$, yet $d(g_E\tau(x_E),g_E\tau(y_E))>\delta_Y$. By replacing $x_E$, resp. $y_E$ by $g_E^{-1}x_E$, resp. $g_E^{-1} y_E$ we may and will assume that $g_E=1_G$. By compactness, there is a net $(E_i)_{i\in I}$ such that $\bigcup_{i\in I} E_i=G$ and $\lim_i x_{E_i}=x$ and $\lim_i y_{E_i}=y$. By continuity and since $(E_i)_i$ cover the group, we have $d(gx,gy)\leq \delta_X$ for every $g\in G$, thus $x=y$ by expansiveness. By continuity, for large enough $i$ we must  have $d(\tau(x_{E_i}),\tau(y_{E_i}))\leq \delta_Y$, a contradiction.
\end{proof}

What follows next is a list of definitions of several specification properties that will be used in this paper.
\begin{definition}
	Let $X$ be a $G$-space. We say that \begin{itemize}[leftmargin=5.5mm]
		\item $X$ has the \emph{weak specification property} if for every $\varepsilon>0$ there exists a finite symmetric $S\subseteq G$ such that for every subsets $A_1,A_2\subseteq G$ and $x_1,x_2\in X$ such that $S\cdot A_1\cap A_2=\emptyset$ there is $y\in X$ satisfying $d(gx_i,gz)<\varepsilon$, $i\in\{1,2\}$ and $g\in A_i$.
		
		\item $X$ has \emph{the pseudoorbit tracing property} (POTP), or \emph{shadowing}, if for every $\varepsilon>0$ there exist $\gamma>0$ and finite $S\subseteq G$ such that for every $G$-indexed set $(x_g)_{g\in G}\subseteq X$ satisfying $d(s x_g, x_{sg})<\gamma$, for all $g\in G$ and $s\in F$ there exists $z\in X$ such that $d(gz,x_g)<\varepsilon$ for every $g\in G$.
		
		\item $X$ has \emph{the strong topological Markov property} if for every $\varepsilon>0$ there exist $\gamma>0$ and finite $S\subseteq G$ such that for every $x,y\in X$  and every finite $A\subseteq G$ satisfying $d_{SA\setminus A}(x,y)<\gamma$ there exists $z\in X$ such that $d_A(z,x)<\varepsilon$ and $d_{G\setminus A}(z,y)<\varepsilon$.
		
		\item $X$ has \emph{the uniform strong topological Markov property} if for every $\varepsilon>0$ there exist $\gamma>0$ and finite $S\subseteq G$ such that for every subset $A\subseteq G$ and every $V\subseteq G$ with $S\cdot A\cdot v_1\cap S\cdot A\cdot v_2=\emptyset$, for $v_1\neq v_2\in V$, and every $V$-indexed set $(x_v)_{v\in V}\subseteq X$ and an element $y\in X$ satisfying $d_{(SA\setminus A)v} (x_v,y)<\gamma$, for all $v\in V$, there exists $z\in X$ such that $d_{Av} (x_v,z)<\varepsilon$, for all $v\in V$, and $d_{G\setminus AV}(z,y)<\varepsilon$.
	\end{itemize}

\end{definition}
\begin{remark}
The weak specification in this generality for general group actions was defined in \cite{NPLi}, the pseudo-orbit tracing property (or shadowing) in this generality in \cite{OsTi}. The topological Markov properties were introduced in the recent \cite{BaFRLiarxiv}.
\end{remark}

\begin{lemma}\label{lem:POTP}
	Let $G$ act expansively on a compact metric space $X$ with an expansiveness constant $\delta>0$. Then $X$ has the pseudoorbit tracing property if and only if there is a finite set $F\subseteq G$ such that for every $G$-indexed set $(x_g)_{g\in G}$ satisfying $d(s x_g,x_{sg})\leq \delta/2$ for every $g\in G$ and $s\in F$ there is $z\in X$ such that $d(gx,x_g)\leq \delta/2$ for every $g\in G$.
\end{lemma}
\begin{proof}
	Suppose that $X$ has the POTP. Apply the definition for $\delta/2$ as $\varepsilon$ in order to get some $\gamma$ and a finite set $S\subseteq G$. Let $D_\gamma$ be a finite symmetric set from Lemma~\ref{lem:keyexpansivefact} and set $F:=D_\gamma\cdot S$ which we claim to be the desired finite set. Indeed, let $(x_g)_{g\in G}\subseteq X$ be a $G$-indexed set satisfying $d(t x_g,x_{tg})\leq \delta/2$ for every $g\in G$ and $t\in F$. We claim that for every $g\in G$ and $s\in S$ we have $d(s x_g,x_{sg})\leq \gamma$ which will be enough. Fix $g\in G$ and $s\in S$. Then for every $t\in D_\gamma$ we have \[d((ts)x_g,t x_{sg})\leq d((ts)x_g,x_{tsg})+d(x_{tsg},t x_{sg})\leq \delta/2+\delta/2,\] which by definition of $D_\gamma$ implies that $d(s x_g,x_{sg})\leq \gamma$ as needed.\medskip
	
	Now conversely suppose $X$ satisfies the condition from the statement and let us show it has the POTP. Fix $\varepsilon>0$. We claim that $\gamma=\delta/2$ and $S=D_\varepsilon\cdot F$ are as desired. Let $(x_g)_{g\in G}\subseteq X$ be such that $d(s x_g,x_{sg})\leq \delta/2$ for every $g\in G$ and $s\in S$. Since $F\subseteq S$ we get that there is $z\in X$ satisfying $d(gz,x_g)\leq \delta/2$ for every $g\in G$. We claim that in fact $d(gz,x_g)\leq \varepsilon$ is true. Indeed, fix $g\in G$. We check that for every $t\in D_\varepsilon$ we have $d((tg)z,t x_g)\leq\delta$ which will prove the claim. So take $t\in D_\varepsilon$,  we have \[d((tg)z,t x_g)\leq d((tg)z,x_{tg})+d(x_{tg},t x_g)\leq \delta/2+\delta/2,\] as desired.
\end{proof}

The following two lemmas are proved analogously as Lemma~\ref{lem:POTP} and the proofs are left to the reader.
\begin{lemma}\label{lem:WSP}
	Let $G$ act expansively on a compact metric space $X$ with an expansiveness constant $\delta>0$. Then $X$ has the weak specification property if and only if there is a finite symmetric set $F\subseteq G$ such that for every subsets $A_1,A_2\subseteq G$ and $x_1,x_2\in X$ such that $A_1\cdot F\cap A_2=\emptyset$ there is $y\in X$ satisfying $d(gx_i,gz)<\delta/2$, $i\in\{1,2\}$ and $g\in A_i$.
\end{lemma}

\begin{lemma}\label{lem:sTMP}
	Let $G$ act expansively on a compact metric space $X$ with an expansiveness constant $\delta>0$. Then $X$ has the strong topological Markov property if and only if there is a finite symmetric set $F\subseteq G$ containing the unit such that for every finite set $A\subseteq G$, every $V\subseteq G$ with $v_1\notin FAv_2$, for $v_1\neq v_2\in V$ and every $V$-indexed set $(x_v)_{v\in V}\subseteq X$ and an element $y\in X$ satisfying $d_{(FA\setminus A)v} (x_v,y)<\delta/2$, for all $v\in V$, there exists $z\in X$ such that $d_{Av} (x_v,z)<\delta/2$, for all $v\in V$, and $d_{G\setminus AV}(z,y)<\delta/2$.
\end{lemma}

\begin{remark}
	If the expansiveness constant $\delta>0$ of a $G$-space $X$ is fixed and clear from the context, we can then the finite set $F\subseteq G$ from the statement of Lemma~\ref{lem:POTP} call a \emph{pseudoorbit tracing set}. It also follows from Lemma~\ref{lem:POTP} that if $X$ has the POTP and $F\subseteq G$ is a pseudoorbit tracing set for $X$, then if $(x_G)_{g\in G}\subseteq X$ is a pseudoorbit for $D_\varepsilon\cdot F$, meaning that $d(tx_g,x_{tg})\leq \delta/2$ for every $g\in G$ and $t\in D_\varepsilon\cdot F$, then there is $z\in X$ satisfying $d(gz,x_g)\leq \varepsilon$ for every $g\in G$.
	
	Analogously, we can call the finite set $F\subseteq G$ from Lemma~\ref{lem:WSP} a \emph{weak specification set}. If we need the weak specification with particular $\varepsilon>0$, the weak specification set $F$ can be replaced by $D_\varepsilon\cdot F$.
	
	Moreover, we can call the finite set $F\subseteq G$ from Lemma~\ref{lem:sTMP} a \emph{strong topological Markov set}. If we need this property with a specific $\varepsilon$, the strong topological Markov set $F$ can be replaced by $D_\varepsilon\cdot F$.
\end{remark}

\section{Garden of Eden theorem}
In this section, we prove one of the two main results of the paper and we establish the Garden of Eden theorem for continuous and expansive actions of amenable groups on compact metrizable spaces satisfying the strong topological Markov and weak specification properties.

We shall need the definition of a tiling of a group.
\begin{definition}
Let $G$ be a group and $A\subseteq B\subseteq G$ be two finite subsets. We say that a subset $T\subseteq G$ is a \emph{$(A,B)$-tiling of $G$} if 
\begin{itemize}
	\item for $s\neq t\in T$, the sets $A\cdot t$ and $A\cdot s$ are disjoint and,
	\item $\bigcup_{t\in T} B\cdot t=G$.
\end{itemize}

\end{definition}
\begin{theorem}\label{thm:GOE}
	Let $G$ be a countable amenable group and $X$ an expansive and compact $G$-metric space satisfying the weak specification property and the strong topological Markov property. Then $X$ satisfies the Moore-Myhill property.
\end{theorem}
\begin{proof}
	Fix an expansiveness constant $\delta>0$.
	Let us show the Moore property. Let $\tau:X\rightarrow X$ be a surjective continuous $G$-equivariant map. We show it is pre-injective. Suppose the contrary and choose $x\sim y\in X$ satisfying $\tau(x)=\tau(y)$. Set $\Delta:=\Delta_\delta(x,y)$, which is by Corollary~\ref{cor:Deltahomoclinic} finite and we may assume it is non-empty and contains $1_G$. Also, let $S'\subseteq G$ be a finite symmetric memory set for $\tau$. Finally, let $S\subseteq G$ be a finite symmetric subset containing $1_G$ such that
	\begin{enumerate}
		\item $S$ is a weak specification and strong topological Markov set for $X$;
		\item $(S')^2\subseteq S$;
		\item $\forall u\in X\forall g\in  S'\cdot D_{\delta/8}\cdot\Delta\; (d_S(gu,gy)\leq \delta\Rightarrow d(\tau(gu),\tau(gx))\leq\delta/2).$
	\end{enumerate}
	
	Now set $E:=D_{\delta/8}\cdot S\cdot D_{\delta/8}\cdot\Delta$. Since all of the finite sets involved in the product may be assumed to be symmetric and containing $1_G$, we expect $E$ to satisfy the same. By \cite[Proposition 5.6.3]{CSCo-book} there exists an $(E^2,E^4)$-tiling of $G$, denoted by $T$. Let $Z$ be the closed (not necessarily $G$-invariant) subset \[\{z\in X\colon \forall t\in T\; (d_{E^2 t}( z, t^{-1}x)\geq \delta/8)\}.\]

	\begin{claim}\label{claim:GOE1}
	We have $\tau[Z]=\tau[X]$.
	\end{claim}
	That is, for every $w\in X$ there exists $z\in Z$ satisfying $\tau(w)=\tau(z)$. Fix $w\in X$. What follows is the proof of the existence of such $z$ that will be divided into three steps. \medskip
	
	\noindent\textbf{Step 1.} Finding $z$ using the strong topological Markov property.\\ \\

	We define a $T$-indexed set $(z_t)_{t\in T}$ in the following way. For $t\in T$ we set
	\[z_t:=\begin{cases}t^{-1}y & \text{if }d_{E^2 t}(w,t^{-1}x)<\delta/4,\\
	w & \text{otherwise.}
	
	\end{cases}\]
	We want to use the strong topological Markov property with the constant $\delta/8$, thus with the set $P:=D_{\delta/8}\cdot S$, for $A=E$, for $V=T$, where the $T$-indexed set is $(z_t)_{t\in T}$, and for $y=w$. Since $T$ is a $(E^2,E^4)$-tiling and $P\subseteq E$, we have that $P\cdot E\cdot t_1\cap P\cdot E\cdot t_2=\emptyset$, for every $t_1\neq t_2\in T$. So we need to check that for every $t\in T$ we have $d_{(PE\setminus E)t}(z_t,w)<\delta/2$. Fix $t\in T$. If $d_{E^2 t}(w,t^{-1}x)\geq \delta/4$, then $z_t=w$ and the inequality $d_{(PE\setminus E)t}(z_t,w)<\delta/2$ is obvious. So suppose that $d_{E^2 t}(w,t^{-1}x)<\delta/4$ and therefore $z_t=t^{-1}y$. Since $d_{G\setminus D_{\delta/8}\cdot\Delta}(x,y)<\delta/8$ and $D_{\delta/8}\cdot\Delta\subseteq E$ we have $d_{(PE\setminus E)t}(t^{-1}x,t^{-1}y)<\delta/8$. Since $P\subseteq E$ we also have by the assumption $d_{(PE\setminus E)t} (w,t^{-1}x)<\delta/4$, so by the triangle inequality, we get $d_{(PE\setminus E)t}(z_t,w)<\delta/2$ as desired.
	
	By strong topological Markov property, it follows that there exists $z\in X$ such that $d_{Et}(z,z_t)<\delta/2$, for all $t\in T$, and $d_{G\setminus FT}(z,w)<\delta/2$.\medskip

	\noindent\textbf{Step 2.} We claim that $z\in Z$.\\ \\
	Suppose the contrary. Then there is $t\in T$ such that $d_{E^2 t}(z,t^{-1}x)<\delta/8$. Either we have $z_t=w$, if $d_{E^2t}(w,t^{-1}x)\geq \delta/4$, in which case $d_{E^2t}(z,w)\leq \delta/8$, so by the triangle inequality $d_{E^2t}(z,t^{-1}x)\geq \delta/8$, a contradiction.
	
	Or we have $z_t=t^{-1}y$ if $d_{E^2t}(w,t^{-1}x)<\delta/4$, in which case $d_{Et}(z,t^{-1}y)\leq \delta/8$. Since $d_{Et}(t^{-1}x,t^{-1}y)>\delta$, by the triangle inequality we have \[\delta/2<d_{Et}(z,t^{-1}x)\leq d_{E^2t}(z,t^{-1}x),\] another contradiction.\medskip

	\noindent\textbf{Step 3.} We claim that $\tau(z)=\tau(w)$.\\ \\
	Suppose again the contrary. Then there exists $g\in G$ such that \[d(g\tau(w),g\tau(z))=d(\tau(gw),\tau(gz))>\delta.\] Define $V\subseteq T$ to be the set $\{t\in T\colon d_{E^2 t}(w,t^{-1}x)<\delta/4\}$. We have two cases. Either there exists $t\in V$ such that $S'g\cap Et\neq\emptyset$, or not. In the latter case, we have that for every $s\in S'$ \[d((sg)z,(sg)w)\leq  \delta/8,\] thus since $S'$ is a memory set for $\tau$, we get $d(\tau(gz),\tau(gw))\leq\delta$, a contradiction.
	
	So suppose that there is $t\in V$ such that $S'g\cap Et\neq\emptyset$. Then we distinguish three more subscases:
	\begin{itemize}
		\item either $S'g\subseteq Et$, however, $S'g\cap  D_{\delta/8}\cdot\Delta t=\emptyset$;
		\item or $S'g\cap  D_{\delta/8}\cdot\Delta t\neq\emptyset$;
		\item or $S'g\nsubseteq Et$.
	\end{itemize}
	In the first case, we have for every $s\in S'$ \[d((sg)z,(sgt^{-1})y)\leq \delta/8,\]and since $S'g\cap  D_{\delta/8}\cdot\Delta t=\emptyset$, we also have for every $s\in S'$ \[d((sgt^{-1})y,(sgt^{-1})x)<\delta/8),\]so for every $s\in S'$ we have \[d((sg)z,(sgt^{-1})x)<\delta/4.\] Combining this with the fact that, since $t\in V$, that for every $s\in S'$ \[d((sg)w,(sgt^{-1})x)<\delta/4,\] we finally obtain that for all $s\in S'$ \[d((sg)z,(sg)w)<\delta/2,\] thus, since $S'$ is a memory set for $\tau$, $d(\tau(gz),\tau(gw))\leq\delta$, again a contradiction.
	
	In the second case, we have $g\in S'\cdot D_{\delta/8}\cdot\Delta t$. Then, since $S\cdot g\subseteq Et$, we have \[d_S(gz,gt^{-1} y)\leq \delta,\] since for every $s\in S$ we have $d((sg)z,(sgt^{-1})y)\leq \delta/8$, so by $(3)$ we get \[d(\tau(gz),\tau(gx))\leq \delta/2.\] On the other hand, since $d_{D_{\delta/8}\cdot S'}(gw,(gt^{-1})x)<\delta/4$ and $S'$ is a memory set for $\tau$ we get \[d(\tau(gw),\tau(gx))\leq \delta/8.\] So finally, by triangle inequality, we get \[d(\tau(gz),\tau(gw))\leq \delta,\] which is again a contradiction.
	
	For the last third case, pick $s\in S'$. If $sg\notin Et$, then $d((sg)z,(sg)w)\leq \delta/8$. If $sg\in Et$, then we have $d((sg)z,(sgt^{-1})y)\leq\delta/8$. On the other hand, since there is by assumption $s'\in S'$ such that $s'g\notin Et$ we have $sg\notin D_{\delta/8}\cdot\Delta t$ since otherwise $s's^{-1}sg\in Et$ as $s's^{-1}\in S$ and $S\cdot D_{\delta8}\cdot\Delta\subseteq E$. It follows that $d((sgt^{-1})y,(sgt^{-1})x)<\delta/8$ and so by the triangle inequality \[d((sg)z,(sgt^{-1})x)<\delta/4.\] Combining this with the fact that $d((sg)w,(sgt^{-1})x)<\delta/4$ we finally obtain by another application of the triangle inequality that \[d((sg)z,(sg)w)<\delta/2.\]
	So again, since $S$ is a memory set for $\tau$ we get $d(\tau(gz),\tau(gw))\leq\delta$, a contradiction. This finishes the proof of Claim~\ref{claim:GOE1}.
	
	\begin{claim}\label{claim:GOE2}
	We have $h_\mathrm{top}(Z)<h_\mathrm{top}(X)$.	
	\end{claim}
	
	First notice that for every closed subset $Y\subseteq X$, $\gamma>0$ and every $E\subseteq F\subseteq G$, where $E,F$ are finite, we have 
	\begin{equation}\label{eq:nets}
	\sep(Y,\gamma,F)\leq \sep(Y,\gamma,F\setminus E)\cdot \sep(Y,\gamma/2,E).
	\end{equation}
	Indeed, suppose the contrary and let $N\subseteq Y$ be a $\gamma$-separated subset of $(Y,d_F)$ of size strictly bigger than $\sep(Y,\gamma,F\setminus E)\cdot \sep(Y,\gamma/2,E)$. Let $M\subseteq N$ be a maximal $\gamma/2$-separated subset of $N$ with respect to $d_E$. Since $|N|>\sep(Y,\delta,F\setminus E)\cdot \sep(Y,\delta/2,E)$, while $|M|\leq \sep(Y,\gamma/2,E)$ there must exists $m\in M$ and set $A\subseteq N$ of size strictly bigger than $\sep(Y,\gamma, F\setminus E)$ such that for every $a\in A$ we have $d_E(a,m)\leq \gamma/2$. It follows, by triangle inequality, that for every $a,b\in A$ we have $d_E(a,b)\leq \gamma$. On the other hand, for every $a\neq b\in A$ we have $d_F(a,b)>\gamma$, so there must be $e\in F\setminus E$ such that $d(ea,eb)>\gamma$. It follows that $A$ is $\gamma$-separated for $d_{F\setminus E}$ which contradicts that $|A|>\sep(Y,\gamma, F\setminus E)$.
	
	Now fix $j\in\Nat$. Enumerate $T_j$ as $\{t_1,\ldots,t_n\}$. For every $1\leq i\leq n$ define \[Z_i:=\{z\in X\colon \forall j\leq i\; (d_{E^2 t_j}(z,t_j^{-1} x)\geq \delta/8)\}.\] Moreover, set $Z_0:=X$. Next we claim that for every $1\leq i\leq n$ we have
	\begin{equation}\label{eq:sepinequality}
	\sep(Z_i,F_j,\delta/4)\leq \sep(Z_{i-1},F_j,\delta/4)-\sep(Z_{i-1},F_j,\delta/4)\cdot \Sigma_{i-1},
	\end{equation}
	where $\Sigma_{i-1}:=1/\sep(Z_{i-1},E t_i,\delta/2)$. To show that, pick an arbitrary $\delta/2$-separated set $N\subseteq Z_{i-1}$ with respect to $d_{F_j\setminus E t_i}$. By the weak specification property (recall that $S$ is a weak specification set) and since $E$ contains $D_{\delta/8}\cdot S\cdot\Delta$ for each $w\in N$ we can find $w_x\in X$ such that
	\begin{itemize}
		\item for every $e\in F_j\setminus E t_i$ we have $d(e w, e w_x)<\delta/8$;
		\item for every $g\in\Delta t_i$ we have $d((g t_i^{-1}) x, g w_x)<\delta/8$.
	\end{itemize}
	In particular, it follows that $N_x:=\{w_x\colon w\in N\}\subseteq \big((X\setminus Z_i)\cap Z_{i-1}\big)$ and that $N_x$ is $\delta/4$-separated with respect to $d_{F_j\setminus E t_i}$. Since by \eqref{eq:nets} we have $\sep(Z_{i-1}, F_j\setminus E t_i,\delta/2)\geq \sep(Z_{i-1},F_j,\delta/4)\cdot\Sigma_i$ we obtain \[\begin{split}\sep(Z_i,F_j,\delta/4) & \leq \sep(Z_{i-1},F_j,\delta/4)-\sep(Z_{i-1},F_j\setminus E t_i,\delta/2)\\ & \leq \sep(Z_{i-1},F_j,\delta/4)-\sep(Z_{i-1},F_j,\delta/4)\cdot\Sigma_i),\end{split}\] which verifies \eqref{eq:sepinequality}.
	
	If we show that for every $i< n$ we have $\Sigma_i\leq \Sigma:=\sep(X,E,\delta/2)$, then we get from \eqref{eq:sepinequality} that \[\begin{split}\sep(Z,F_j,\delta/4) & \leq \sep(Z_n, F_j,\delta/4) \leq \prod_{i=0}^{n-1}(1-\Sigma_i)\sep(Z_0,F_j,\delta/4)\\ & \leq (1-\Sigma)^{|T_j|}\sep(X,F_j,\delta/4).\end{split}\]
	But that is clear since for every $i< n$, if $M\subseteq Z_i$ is a $\delta/2$-separated set with respect to $d_{E t_{i+1}}$, then $t_{i+1}^{-1} M$ is a $\delta/2$-separated set in $X$ with respect to $d_E$.\medskip
	
	Finally, using that for all sufficiently large $n\in\Nat$ we have by \cite[Proposition 5.6.4]{CSCo-book} $|T_n|\geq K|F_n|$, for some constant $K>0$, we get \[\begin{split}\log \sep(Z,d_{F_n},\delta/4) & \leq \log (1-\Sigma)^{|T_n|}\sep(X,F_n,\delta/4)\\ & =|T_n|\log (1-\Sigma)+\log \sep(X,F_n,\delta/4)\\ & \leq K|F_n|\log(1-\Sigma)+ \log \sep(X,F_n,\delta/4),\end{split}\] 
	
	therefore we get 
	
	\[\begin{split}h_\mathrm{top}(Z,G) & =h_\sep(Z, \delta/4,d)=\limsup_{n\to\infty} \frac{1}{|F_n|}\log \sep(Z,d_{F_n},\delta/4)\\ & \leq \limsup_{n\to\infty} \frac{1}{|F_n|}\Big(K|F_n|\log(1-\Sigma)+ \log \sep(X,F_n,\delta/4)\Big)\\ & = \limsup_{n\to\infty} \Big(K\log(1-\Sigma)+ \frac{1}{|F_n|}(\log \sep(X,F_n,\delta/4))\Big)\\ & = K\log(1-\Sigma)+h_\sep(X,\delta/4,d)\\ & < h_\mathrm{top}(X,G),\end{split}\] which finishes the proof of Claim~\ref{claim:GOE2}.
	
	\begin{claim}\label{claim:GOE3}
	We have $h_\mathrm{top}(\tau[X])<h_\mathrm{top}(X)$.
	\end{claim}
	
	Since we already proved that $h_\mathrm{top}(Z)<h_\mathrm{top}(X)$ and that $\tau[Z]=\tau[X]$ it is enough to show that $h_\mathrm{top}(\tau[Z])\leq h_\mathrm{top}(Z)$. By Proposition~\ref{prop:entropyexpansive}, it is enough to show that $h_\sep(\tau[Z],\delta,d)\leq h_\sep(Z,\delta,d)$. Let $N\subseteq \tau[Z]$ be a $\delta$-separated set with respect to $d_F$ for some $F\subseteq G$. For each $x\in N$ let $x'\in Z$ be any element from $\tau^{-1}(x)$ and let $N':=\{x'\colon x\in N\}\subseteq Z$. We claim that $N'$ is $\delta$-separated for $d_{S\cdot F}$. Indeed, this follows since $S$ is a memory set for $\tau$ as follows. If $x\neq y\in N$, then there is $e\in F$ such that $d(ex,ey)>\delta$. If there were not $s\in S$ such that $d((se)x',(se)y')>\delta$, then by the definition of a memory set and the fact that $S$ is symmetric we would have that $d(ex,ey)\leq \delta$, a contradiction. It follows that \[\begin{split}h_\sep(Z,\delta,d) & = \limsup_{n\to\infty} \frac{1}{|S\cdot F_n|} \log \sep(Z,d_{S\cdot F_n},\delta)\\ & \leq \limsup_{n\to\infty} \frac{1}{|F_n|} \log \sep(Z,d_{S\cdot F_n},\delta)\\ & \leq \limsup_{n\to\infty} \frac{1}{|F_n|} \log \sep(\tau[Z],d_{F_n},\delta)= h_\sep(\tau[Z],\delta,d),\end{split}\]
	and we are done.\bigskip
	
	For the sake of completeness, althought this has been already shown by H. Li in \cite{Li19}, let us prove the Myhill property under the weaker assumption that the $G$-space $X$ satisfies the weak specification property. We again assume that the expansiveness constant is $\delta>0$.
	
	Let $\tau:X\rightarrow X$ be a continuous $G$-equivariant map and suppose that it is not surjective. We shall show that it is not pre-injective. Let $S\subseteq G$ be a finite symmetric set containing $1_G$ such that
	\begin{enumerate}
		\item there is a memory set $S'\subseteq G$ such that $D_{\delta/4}\cdot S'\subseteq S$;
		\item there is a weak specification set $S''\subseteq G$ for $X$ such that $D_{\delta/4}\cdot S''\subseteq S$.
	\end{enumerate}
	
	Set $Y:=\tau[X]\subseteq X$. Since $Y\neq X$, arguing exactly as in the proof of Claim~\ref{claim:GOE2} and using Proposition~\ref{prop:entropyexpansive} we get that $h_{\rm{top}}(Y,G)<h_{\rm{top}}(X,G)$ and that moreover, for a given F\o lner sequence $(F_n)_n$ (resp. net in case $G$ is uncountable) there exists $j\in\Nat$ such that \[\sep(Y,S^2\cdot F_j,\delta)<\sep(X,F_j,\delta).\]
	Fix $x\in X$. Let $X_0$ be the set $\{z\in X\colon d_{G\setminus S\cdot F_j}(x,z)\leq \delta\}$. By Corollary~\ref{cor:Deltahomoclinic}, $X_0\subseteq [x]_\sim$. Let $N\subseteq X$ be a $\delta$-separated set with respect to $d_{F_j}$ of size $\sep(X,F_j,\delta)$. Using the weak specification property with respect to $D_{\delta/4}\cdot S''\subseteq S$, for each $y\in N$ we can find $z_y\in X_0$ such that $d_{F_j}(y,z_j)\leq \delta/4$. In particular, we get \[\sep(X_0,F_j,\delta/2)\geq \sep(X,F_j,\delta).\]
	Since $S$ is also a memory set for $\tau$ we get that for every $z\in X_0$
	\begin{equation}\label{eq:myhilleq}
	d_{G\setminus S^2\cdot F_j}(\tau(z),\tau(x))\leq \delta.
	\end{equation}
	Therefore, as $\sep(X_0,F_j,\delta/2)>\sep(Y,S^2\cdot F_j,\delta)$ there must exist $y,y'\in X_0$ such that $d_{F_j}(y,y')>\delta/2$, in particular $y\neq y'$, yet $d_{S^2\cdot F_j}(\tau(y),\tau(y'))\leq \delta$. Combining with \eqref{eq:myhilleq}, we have that for every $g\in G$ \[d(g\tau(y),g\tau(y'))\leq \delta,\] thus $\tau(y)=\tau(y')$ by expansiveness. Since $y,y'\in X_0\subseteq [x]_\sim$, we have $y\sim y'$ and we are done.
\end{proof}

Both Moore and Myhill properties make sense to be investigated only for expansive actions of amenable groups. Indeed, it was shown by Bartholdi in \cite{Ba10}, resp. in \cite{Ba19} that the Moore, resp. Myhill property fails already for a full topological shift of any non-amenable group. However, there is a well-known weakening of the Myhill property, called \emph{surjunctivity}, which asks whether an injective continuous $G$-equivariant map of some dynamical system is surjective. It is a famous conjecture of Gottschalk from \cite{Gott} whether for every group $G$ and finite alphabet $A$, the full shift $A^G$ is surjunctive. This conjecture has received a considerable attention especially since Gromov's introduction of sofic groups and his proof that such groups do satisfy Gottschalk's conjecture (see \cite{Grom}). A very recent result of Ceccherini-Silberstein, Cornaert, and Li establishes surjunctivity of expansive actions of sofic groups on compact metrizable spaces having the weak specifiaction and strong topological Markov properties, and having non-negative sofic entropy (see \cite{CSCoLi21}).

It is only very recently when an analogous weakening has been introduced also for the Moore property. This is the \emph{dual surjunctivity} of Capobianco, Kari, and Taati from \cite{CaKaTa} (see also the recent \cite{DoGi} for additonal results, including those concerning more general expansive dynamical systems), which has been also proved for full shifts of sofic groups. With respect to the recent result \cite{CSCoLi21}, it is natural to ask whether all expansive actions of sofic groups on compact metrizable spaces satisfying certain additional properties, as in \cite{CSCoLi21}, are dual surjunctive.

\subsection{Beyond countable groups and metrizable spaces}\label{subsect:nonmetrizable}
Although Theorem~\ref{thm:GOE} was stated for countable groups, one can check that the proof can be easily adapted to handle the uncountable case as well, replacing the limits of sequences by limits of nets - in appropriate places. Here we hint how to remove the assumption of metrizability of the compact space $X$ from Theorem~\ref{thm:GOE}. First, we state the following general definition of expansiveness of an action, not using any metric. It is an exercise that this coincides with Definition~\ref{def:expansiveness} when $X$ is equipped with a compatible metric.
\begin{definition}\label{def:generalexpansiveness}
Let $G$ be an arbitrary group acting continuously on a compact topological space $X$. The action is \emph{expansive} if there exists a finite open cover $(U_i)_{i=1}^n$ of $X$ such that the sets $\{gU_i\colon g\in G, i\leq n\}$ form a subbase of the topology of $X$.
\end{definition}
A general definition of expansive actions of groups on uniform spaces was given in \cite{CSCo11}. It is again an exercies that it coincides with the definition above.\medskip

The following are well known fact from general topology (we shall refer to the book \cite{Engelking}).
\begin{itemize}
	\item The topology of $X$ is induced by a unique uniform structure $\UU$ (see e.g. \cite[Theorem 8.3.13]{Engelking}).
	\item For every $U\in\UU$ there is a uniformly continuous pseudometric $\rho$ on $X$ such that $\rho(x,y)<1$ for every $(x,y)\in U$ (see e.g. \cite[Corollary 8.1.11]{Engelking}).
	\item For every two uniformly continuous pseudometrics $\rho_1$ and $\rho_2$ on $X$, also $\max \{\rho_,\rho_2\}$ is a uniformly continuous pseudometric (this is obvious).
\end{itemize}

Since the finite open cover from Definition~\ref{def:generalexpansiveness} can be replaced by any finite open subcover, it follows from the general facts above that if a group $G$ acts expansively on a compact space $X$ then there exists a single uniformly continuous pseudometric $\rho$ on $X$ and a constant $\delta>0$ such that for every $x\neq y$ there is $g\in G$ so that $\rho(gx,gy)>\delta$.

If $\tau:X\rightarrow X$ is continuous and $G$-equivariant, Definition~\ref{def:memoryset} and Lemma~\ref{lem:memoryset} are adapted in a straightforward way to this pseudometric $\rho$. The topological entropy of the action on $X$ can be also computed using this pseudometric $\rho$. This follows from the fact that $\rho$ is dynamically generating in the sense of \cite[Definition 9.35]{KeLi-book} and then applying \cite[Theorem 9.38]{KeLi-book}. Note that the argument there is for metrizable compact space, however can be also used for non-metrizable spaces by replacing the genuine metric from \cite[Lemma 9.37]{KeLi-book} by the pseudometric $\max\{\rho,\rho'\}$, where $\rho'$ is a uniformly continuous pseudometric generating the open sets of the finite open cover (which varies) from \cite[Lemma 9.37]{KeLi-book}.

The other notions such as the strong topological Markov property, the weak specification property and the homoclinic relation are defined as in the metrizable case with the exception that their definition starts with the quantifier ``for every uniformly continuous pseudometric $\lambda$" and the rest of the definition is with respect to $\lambda$. The set $\Delta_\delta(x,y)$ is defined with respect to the distinguished pseudometric $\rho$ and Corollary~\ref{cor:Deltahomoclinic} is still valid.

We can then repeat verbatim the whole proof of Theorem~\ref{thm:GOE}, using the pseudometric $\rho$ instead of the metric $d$.

\section{Weakly periodic points}
The goal of this section is to prove the second main result, that certain expansive actions of finitely generated groups with at least two ends have weakly periodic points, generalizing the result from \cite{Cohen}.

We shall need another variant of topological Markov property, different from those introduced in \cite{BaFRLiarxiv}, although it may turn out it is equivalent to one of them. We formulate the following definition only for expansive actions, since our work is limited to them, however a straightforward modification allows to formulate it generally.
\begin{definition}
	Let a group $G$ act continuously and expansively on a compact metric space $X$ with an expansiveness constant $\delta>0$. We say that the action has \emph{the cover strong topological Markov property} if there is a finite symmetric set $S\subseteq G$ containg the unit such that for every cover $(F_n)_{n\in\Nat}$ of $G$ by disjoint, not necessarily finite, sets and for every sequence of elements $(x_n)_{n\in\Nat}\subseteq X$ satisfying that for every $n\neq m\in\Nat$ and every $g\in S\cdot F_n\cap F_m$, \[d(gx_n,gx_m)\leq \delta/2,\] then there exists $z\in X$ with the property that for every $n$ and $g\in F_n$ we have \[d(gz,gx_n)\leq \delta/2.\]
\end{definition}
The next result says that this topological Markov property lies, in principle, somewhere in between the pseudo-orbit tracing property and the (uniform) strong topological Markov property, but we conjecture that it might be actually equivalent to the latter.
\begin{proposition}
Let a group $G$ act continuously and expansively on a compact metric space $X$ with an expansiveness constant $\delta>0$. If the action has the pseudo-orbit tracing property, then it has the cover strong topological Markov property, and if it has the latter, then it has the uniform strong topological Markov property.
\end{proposition}
\begin{proof}
Let us fix the group $G$, the space $X$, and the action of $G$ on $X$. Suppose first it has the pseudo-orbit tracing property and let us show it has the cover strong topological Markov property. Let $S$ be a finite pseudo-orbit tracing set for the action, assuming as usual that it is symmetric and containing the unit. Let $(F_n)_{n\in\Nat}$ be a cover of $G$ by disjoint sets and let $(x_n)_{n\in\Nat}\subseteq X$ be a sequence of elements satisfying that for every $n\neq m\in\Nat$ and every $g\in S\cdot F_n\cap F_m$, \[d(gx_n,gx_m)\leq \delta/2.\] We produce a $G$-indexed set $(y_g)_{g\in G}$ as follows. For any $g\in G$ we set \[y_g:= g x_n\; \text{ if }g\in F_n.\] Let us verify that $(y_g)_g$ is a pseudo-orbit for the set $S$. Pick any $g\in G$ and $s\in S$. We need to check that $d(s y_g,y_{sg})<\delta/2$. Let $n\in\Nat$ be such that $g\in F_n$. If $sg\in F_n$, then $s y_g$ is, by definition, equal to $y_{sg}$, and there is nothing to prove. So we assume that $sg\in F_m$, where $m\neq n$, and we have $sg\in S\cdot F_n\cap F_m$. It follows we have \[d(s y_g, y_{sg})=d(sg x_n, sg x_m)<\delta/2,\] and the verification is finished. Therefore, there exists $z\in X$ shadowing the pseudo-orbit, i.e. for every $g\in G$ we have $d(gz, y_g)<\delta/2$, consequently for for every $n$ and $g\in F_n$ we have \[d(gz,gx_n)\leq \delta/2,\] and we are done.\medskip

Now assume that the action has the cover strong topological Markov property and we show that it has the uniform strong topological Markov property. Let $S$ be the finite symmetric set from the definition and we show it is also a strong topological Markov set. Let $A\subseteq G$ be a finite set and let $V\subseteq G$ be such that for every $v\neq w\in V$ we have $S\cdot A\cdot v\cap S\cdot A\cdot w-\emptyset$. Let $(x_v)_{v\in V}\subseteq X$ and $y\in X$ be such that $d_{(SA\setminus A)v} (x_v,y)<\delta/2$, for every $v\in V$. Then $\UU:=(A\cdot v)_{v\in V}\cup \{G\setminus A\cdot V\}$ is a cover of $G$ by disjoint sets. Pick $B\neq C\in\UU$ and assume that there is $g\in B$ and $s\in S$ so that $sg\in S\cdot B\cap C$. By the assumption, necessarily $C=G\setminus A\cdot V$, and again by the assumption we get that $d(sg x_v, sg y)<\delta/2$, where $v\in G$ is such that $B=A\cdot v$. By the cover strong topological Markov property there is $z\in X$ so that
\begin{itemize}
	\item for every $v\in V$ and every $g\in A\cdot v$ we have $d(gz, g x_v)<\delta/2$,
	\item for every $g\in G\setminus A\cdot V$ we have $d(gz,gy)<\delta/2$,
\end{itemize}
which is exactly what we were supposed to show.
\end{proof}

We are ready to state the main result of this section. We shall not define here the ends of groups since we will not directly need it. The reader can find it either in Cohen's article \cite{Cohen}, or we refer to \cite[Section 9.1.3]{DruKap-book} for a general treatment and more information.

\begin{theorem}\label{thm:generalizedCohen}
	Let $G$ be a finitely generated group with at least two ends and let $G$ act expansively on a compact metrizable space $X$ so that the action satisfies the strong topological Markov property. Then $X$ has a weakly periodic point.
	
	Moreover, if the action satisfies the weak specification property, then the set of such points is dense in $X$.
\end{theorem}

The following lemma extracts the main geometric group theoretic ingredient from \cite{Cohen} that is sufficient for our generalization.
\begin{lemma}[extracted from the proof of {\cite[Theorem 2.1]{Cohen}}]\label{lem:Cohen}
	Let $S\subseteq G$ be a finite symmetric set containing the unit. There exists an infinite order $g\in G$ such that for any $m\in\Nat$ there is a subset $A_m\subseteq G$ satisfying the following properties.
	\begin{itemize}
		\item The sets $\{A_m\cdot g^{nm}\}_{n\in\Int}$ form a disjoint partition of $G$.
		\item We have $S\cdot A_m\subseteq A_m\cup S^2\cup S^2\cdot g^m$.
		\item If for some $k\neq l\in\Int$ we have $S\cdot A_m\cdot g^{km}\cap A_m\cdot g^{lm}\neq\emptyset$, then $|k-l|=1$ and for any $h\in S\cdot A_m\cdot g^{km}\cap A_m\cdot g^{lm}$ we have either \[h\in S^2\cdot g^{km}\text{ and }hg^m\in A_m\cdot g^{km}\; (\text{when }k=l+1),\] or \[h\in S^2\cdot g^{(k+1)m}\text{ and }hg^{-m}\in A_m\cdot g^{km}\; (\text{when }k=l-1).\]
	\end{itemize}
\end{lemma}
\begin{proof}[Proof of Lemma~\ref{lem:Cohen}.]
Let $S\subseteq G$ be the given finite set as in the statement of the lemma. Fix a word metric $\rho$ on $G$ with respect to some finite symmetric generating set of $G$ and let $n\in\Nat$ be big enough so that $n\geq N_G$, where $N_G$ is a constant depending on $G$ from \cite[Lemma 2.5]{Cohen}, and $S\subseteq B_G(n)$, where $B_G(n)$ is a ball around the unit of $G$ with respect to $\rho$. In fact, we may without loss of generality assume that $S=B_G(n)$.

Let $g'\in G$ be the element obtained by \cite[Lemma 2.5]{Cohen} as $n$-axial (we refer for the terminology to \cite{Cohen}, here we do not need to know its precise meaning). We have that $g'$ has infinite order (see the paragraph after the proof of \cite[Lemma 2.5]{Cohen}) and we set $g$ to be some power of $g'$ so that $S^2\cdot g^k\cap S^2=\emptyset$, exactly as in \cite[p. 609]{Cohen}.

In the rest of the proof we show that $g\in G$ is as desired (we already know it has infinite order) and we fix additionally some $m\in\Nat$. Now we let $A_m$ be equal to the set $\mathcal{S}$ from \cite[Definition 2.6]{Cohen} using the same integer $m$. The first two items of Lemma~\ref{lem:Cohen} then follow from \cite[Lemma 2.7]{Cohen} (notice that the order of multiplication in \cite{Cohen} is reversed).

Finally, we show the last item. So pick some $k\neq l\in\Nat$ such that $S\cdot A_m\cdot g^{km}\cap A_m\cdot g^{lm}\neq\emptyset$. Translating these sets from the right by $g^{-km}$, we can assume that $k=0$. So suppose that there is $h\in S\cdot A_m\cap A_m\cdot g^{lm}$. We have already showed that we also have that $h\in A_m\cup S^2\cup S^2\cdot g^m$. Since $A_m$ and $A_m\cdot g^{lm}$ are disjoint, we must have that $h\in S^2\cup S^2\cdot g^m$.
\begin{enumerate}
	\item If we have that $h\in S^2$, then we are in the situation as in the last item on \cite[p. 610]{Cohen} (notice the typo there: our case here corresponds to the case $x\in B^2\setminus \mathcal{S}$, although there is mistakenly stated $x\in g^m B^2\setminus \mathcal{S}$) and by the argument there we have $hg^m\in A_m$, and so $h\in A_m\cdot g^{-m}$ and $l=-1$.
	\item If we have that $h\in S^2\cdot g^m$ then like in the first item on \cite[p. 611]{Cohen} (the case there when $x\in g^m B^2\setminus \mathcal{S}$) we get that $hg^{-m}\in A_m$, so $h\in A_m\cdot g^m$ and $l=1$.
\end{enumerate}
\end{proof}

\begin{proof}[Proof of Theorem~\ref{thm:generalizedCohen}.]
Let $S\subseteq G$ be a finite symmetric set containing the unit that is a cover strong topological Markov set for the action.

Fix now some compatible metric $d$ on $X$ and suppose the expansiveness constant, with respect to $d$, is some $\delta>0$. Fix also a finite $\delta/4$-dense net $N\subseteq X$. That is, for every $x\in X$ there is $y\in N$ such that $d(x,y)<\delta/4$. Let $\NN$ denote the set of non-empty subsets of $N$. Let $\eta:X\rightarrow \NN$ be the function associating to every $x\in X$ the set of those elements of $N$ that realize the distance of $x$ from $N$; that is, for every $x\in X$ and $y\in \eta(x)$ \[d(x,y)=d(x,N).\] Finally, let $\xi:x\rightarrow \NN^{S^2}$ be the function defined by, for $x\in X$ and $s\in S^2$, \[\xi(x)(s):=\eta(sx).\]

Pick now some arbitrary $x\in X$. Since $\NN^{S^2}$ is a finite set, there must exist $m,m'\in\Nat$ so that $\xi(g^m x)=\xi(g^{m'}x)$. By replacing $x$ by $g^{-m}x$ if necessary, we may without loss of generality assume that in fact there is $m\in \Nat$ so that \[\xi(x)=\xi(g^m x).\] By the definition of the function $\xi$ and by the triangle inequality, we obtain \[d_{S^2}(x,g^m x)<\delta/2.\]

For every $n\in\Int$, set $x_n:=g^{-nm}x$. We wish to apply the cover strong topological Markov property with respect to the partition $\{A_m\cdot g^{nm}\colon n\in\Nat\}$ and the set of elements $(x_n)_{n\in\Int}$ to obtain an element $z\in X$ satisfying \[\forall k\in\Int\; (d_{A_m\cdot g^{km}}(z,x_k)<\delta/2).\] In order to show that, we need to check that for every $k\neq l\in\Int$ and $s\in S$, if $s\cdot A_m\cdot g^{km}\cap A_m\cdot g^{lm}\neq\emptyset$, i.e. there are $a_1,a_2\in A_m$ such that $sa_1g^{km}=a_2g^{lm}$, then \[d((sa_1g^{km})x_k,(a_2g^{lm})x_l)\leq \delta/2.\] By Lemma~\ref{lem:Cohen}, we have that $|k-l|=1$ and 
\begin{enumerate}
	\item either $k=l+1$ and $sa_1g^{km}\in S^2\cdot g^{km}$ and $sa_1g^{(k+1)m}\in A_m\cdot g^{km}$,
	\item or $k=l-1$ and  $sa_1g^{km}\in S^2\cdot g^{(k+1)m}$ and $sa_1g^{(k-1)m}\in A_m\cdot g^{km}$.\bigskip
\end{enumerate}

In the former case, since $d_{S^2}(x,g^m x)<\delta/2$ and $sa_1\in S^2$, we have \[\begin{split}d((sa_1g^{km})x_k,(a_2g^{lm})x_l) &=d((sa_1g^{km}g^{-km})x,(sa_1g^{km}g^{-lm})x)\\ &=d((sa_1)x,(sa_1g^m)x)<\delta/2.\end{split}\]

In the latter case, we have analogously, since $d_{S^2}(x,g^m x)<\delta/2$ and $sa_1g^{-m}\in S^2$,
\[\begin{split}d((sa_1g^{km})x_k,(a_2g^{lm})x_l)&=d((sa_1g^{km}g^{-km})x,(sa_1g^{km}g^{-lm})x)\\ &=d((sa_1)x,(sa_1g^{-m})x)\\ &=d((sa_1g^{-m}g^m)x,(sa_1g^{-m})x)<\delta/2.\end{split}\]\bigskip

It follows that there indeed is $z\in X$ satisfying, as desired, \[\forall k\in\Int\; (d_{A_m\cdot g^{km}}(z,x_k)<\delta/2).\] We show that $g^m z=z$. If not, there is, by expansiveness, some $h\in G$ such that $d(hz,(hg^m)z)>\delta$. Since the sets $(A_m\cdot g^{km})_{k\in\Int}$ form a cover of $G$, there is $k\in\Int$ and $a\in A_m$ such that $h=ag^{km}$. It follows that \[\begin{split}d(hz,(hg^m)z)&\leq d(hz,h x_k)+d(h x_k,(hg^m)x_{k+1})+d((hg^m)x_{k+1},(hg^m)z)\\ &= d(hz,h x_k)+d(a x,a x)+d((hg^m)x_{k+1},(hg^m)z)\\ & \leq \delta/2+\delta/2=\delta,\end{split}\] a contradiction.\bigskip

We now show the `moreover' part of the statement of the theorem. We retain the notation established during the proof.

 Let $W\subseteq G$ be a finite symmetric set containg the unit that is a weak specification set for the action. In order to prove the density, pick some $w\in X$ and $\delta/2>\varepsilon>0$ and we shall find a weakly periodic point $z_w\in X$ such that $d(z_w,w)<\varepsilon$. By Lemma~\ref{lem:keyexpansivefact}, there is a finite set $D_\varepsilon\subseteq G$ such that for every $z'\in X$, if $d_{D_\varepsilon}(z',w)\leq \delta$, then $d(z',w)<\varepsilon$. Using the weak specification property, we can find $z'_w\in X$ such that
 \begin{itemize}
 	\item $d_{D^2_\varepsilon}(z'_w,w)\leq \delta/2$ (thus $d_{D_\varepsilon}(z'_w,w)<\varepsilon$),
 	\item and $d_{G\setminus (D_{\delta/4}\cdot W\cdot D^2_\varepsilon)}(z'_w,z)\leq \delta/4$.
 \end{itemize}
 Let $l\in\Nat$ be such that $D_{\delta/4}\cdot W\cdot D^2_\varepsilon\subseteq \bigcup_{i\in [-l,l]} A_m\cdot g^{im}$. We now set $\bar A:=\bigcup_{i\in [-l,l]} A_m\cdot g^{im}$. We consider the cover of $G$ by the disjoint sets $\{\bar A\cdot g^{lmn}\colon n\in\Int\}$ and we define the elements $(w_n)_{n\in\Int}$ by setting $w_n:=g^{-lmn}z'_w$. We wish to apply the cover strong topological Markov property with this cover of $G$ and this set of elements of $X$. We need to check that for any $n,n'\in\Int$ and any $h\in S\cdot \bar A\cdot g^{lmn}\cap \bar A\cdot g^{lmn'}$, we have $d(hw_n,hw_{n'})\leq \delta/2$. We may assume that $h\notin \bar A\cdot g^{lmn}$, in particular, $hg^{-lmn}\notin D_{\delta/4}\cdot W\cdot D^2_\varepsilon$, and that $h=sag^{lmn}$ for some $s\in S$ and $a\in \bar A$. Then we have \[\begin{split}d(hw_n,hw_{n'}) & \leq  d((sa) z'_w,(sa)z)\\ &+d((sa)z,(sag^{(n-n')lm})z)+d((sag^{(n-n')lm})z, (sag^{(n-n')lm})z'_w)\\ & \leq \delta/4+0+\delta/4=\delta/2,\end{split}\] which verifies the claim. So there exists $z_w\in X$ such that for every $n$ and $h\in \bar A\cdot g^{lmn}$ we have $d(h z_w, (hg^{-lmn})z'_w)\leq \delta/2$.
 
 Then we have \[d_{D_\varepsilon}(z_w,w)\leq d_{D_\varepsilon}(z_w,z'_w)+d_{D_\varepsilon}(z'_w,w)\leq \delta/2+\varepsilon\leq \delta,\] thus $d(z_w,w)<\varepsilon$.
 
 Finally, we claim that $g^{lm}z_w=z_w$. However, the proof is completely analogous to the proof that $g^m z=z$ and is left to the reader.
\end{proof}

Cohen mentions in \cite{Cohen} that he originally conjectured that given a finitely generated group $G$ there exists a subshift of finite type of $G$ that is strongly aperiodic if and only if $G$ is one-ended. This was refuted by Jeandel in \cite{Jea} who showed that groups with undecidable word problem do not admit such subshifts of finite type.

While Cohen's geometric argument can be generalized to more general dynamical systems, as demonstrated above, this is less clear with Jeandel's computability argument which seems to be more closely tight to subshifts. It is therefore plausible, although it would be surprising, that the original Cohen's conjecture is correct within the larger class of expansive dynamical systems satisfying certain topological Markov properties.\bigskip

\noindent\textbf{Acknowledgements.} We would like to thank to Tullio Ceccherini-Silberstein for his comments and for introducing us to topological Markov properties.

\bibliographystyle{siam}
\bibliography{references}
\end{document}